\documentclass{article}
\usepackage[utf8]{inputenc}
\usepackage{amsmath}
\usepackage[utf8]{inputenc}
\usepackage[english]{babel}

\usepackage[arrow,matrix,curve,cmtip,ps]{xy}

\usepackage{amsthm}
\usepackage{comment}
\theoremstyle{definition}
\newtheorem{definition}{Definition}[section]

\theoremstyle{plain}

\newtheorem{theorem}{Theorem}[section]

\newtheorem{corollary}[theorem]{Corollary}

\newtheorem{proposition}[theorem]{Proposition}

\newtheorem{lemma}[theorem]{Lemma}

\theoremstyle{definition}

\newtheorem{remark}[theorem]{Remark}

\newtheorem{example}[theorem]{Example}

\usepackage[utf8]{inputenc}
\usepackage[english]{babel}

\usepackage{amsfonts}

\title{Regular Ideals of Locally-Convex Higher-Rank Graph Algebras}
\author{Tim Schenkel}
\date{\today}

\begin{document}
\maketitle
\begin{abstract}
    We give a vertex set description for basic, graded, regular ideals of locally-convex Kumjian-Pask Algebras. We also show that Condition $(B)$ is preserved when taking the quotient by a basic, graded, regular ideal. We further show that when a locally-convex, row-finite $k$-graph satisfies Condition $(B)$, all regular ideals are graded. We then show the same things hold for gauge-invariant, regular ideals in locally-convex $k$-graph $C^{*}$-algebras.  
\end{abstract}
\section{Introduction}

In this paper we will study the regular ideals of higher-rank graph algebras. We study them first in the algebraic setting: Kumjian-Pask Algebras. We then look at them in the analytic setting of $k$-graph $C^{*}$-algebras.

Higher-rank graph $C^{*}$-algebras were first introduced in \cite{Kumjian Pask} in 2000. In Kumjian and Pask's work, they focus on row-finite $k$-graphs with no sources. They were able to show that the gauge-invariant uniqueness theorem could be generalized from graph $C^{*}$-algebras. Additionally, they were able to find conditions for simplicity. This gave evidence to support the hope that much of the theory from graph $C^{*}$-algebras might carry over. These findings were further generalized in \cite{RaeburnSimsYeend}, where Raeburn, Sims, and Yeend introduced the concept of local-convexity in this setting. The local-convexity condition allows for sources to appear in the graphs. They were able to prove a generalization of the Cuntz-Krieger Uniqueness Theorem and show that there is a lattice isomorphism from the saturated, hereditary sets of vertices to the gauge-invariant ideals of $k$-graph $C^{*}$-algebras.

Much like the creation of  Leavitt path algebras as an algebraic analogue of graph $C^{*}$-algebras, Kumjian-Pask algebras were created as an algebraic analogue of $k$-graph $C^{*}$-algebras. They were first introduced in \cite{Kumjian Pask Algebra}. The authors were able to prove many algebraic analogues to theorems proved in the $C^{*}$-algebra setting, included proofs of the uniqueness theorems and the lattice isomorphism of saturated and hereditary sets of vertices to the basic, graded ideals. These works were again generalized to the locally-convex setting in \cite{ClarkFlynnHuef}.

In the recent work of \cite{BrownFullerPittsReznikoff} and \cite{GoncalvesRoyer}, the regular ideals of row-finite, no-source graph $C^{*}$-algebras and Leavitt-Path algebras respectively were studied. In both of \cite{BrownFullerPittsReznikoff} and \cite{GoncalvesRoyer} a vertex description of the regular, gauge-invariant ideals and regular, graded ideals were found respectively. As directed graphs can be seen as $1$-graphs, the current paper generalizes both into higher-rank and by allowing for sources. In \cite{BrownFullerPittsReznikoff} and \cite{GoncalvesRoyer} it was shown that Condition $(L)$ (a graph satisfies Condition $(L)$ if every cycle has an entry) is preserved when quotienting by regular ideals.

 In this paper we provide clear vertex set descriptions for the basic, graded, regular ideals of Kumjian-Pask algebras, Theorem \ref{3}, and the gauge-invariant regular ideals of $k$-graph $C^{*}$-algebras, Theorem \ref{4}. 
 Condition $(B)$ was introduced in \cite{RaeburnSimsYeend} as a generalization of Condition $(L)$ for locally-convex higher-rank-graphs. For this reason it is a natural fit for replacing Condition $(L)$ in theorems similar to those seen in \cite{BrownFullerPittsReznikoff} and \cite{GoncalvesRoyer}. We see in Theorem \ref{5} and Theorem \ref{6} that Condition $(B)$ is preserved in the graph when quotienting by a regular ideal, which is not the case for arbitrary quotients. 
 Condition $(B)$ was introduced to extend the Cuntz-Krieger Uniqueness Theorem for $k$-graphs with sources \cite{RaeburnSimsYeend}. 
We further show in Corollaries \ref{60} and \ref{61} that when a locally-convex, row-finite $k$-graph satisfies Condition $(B)$ that all regular ideals of the Kumjian-Pask algebra are graded and all regular ideals of the $k$-graph C$^{*}$-algebra are gauge-invariant respectively.

\subsection*{Acknowledgement}
Thank you to my advisor Dr. Adam Fuller for the guidance and help throughout this paper as part of my dissertation.

\section{Background: $k$-Graphs}

The information on $k$-graphs will pertain to both Kumjian-Pask algebras and $k$-graph $C^{*}$-algebras so we include it here in its own section. We use the definitions from \cite{Kumjian Pask}, \cite{Sims} and \cite{ClarkFlynnHuef}. 
\begin{definition}
 Let $k \in \mathbb{N}\setminus \{0\}$. A $k$-graph is a pair $(\Lambda , d)$ where $\Lambda$ is a countable category and $d$ is a functor from $\Lambda$ to $\mathbb{N}^{k}$ which satisfies the factorization property: for all $\lambda \in Mor(\Lambda)$ and all $m,n\in \mathbb{N}^{k}$ such that $d(\lambda)=m + n$, there exists unique morphisms $\mu, \nu$ in $Mor(\Lambda)$ such that $d(\mu)=m$, $d(\nu)=n $ and $\lambda = \mu \nu$. 
 \end{definition}
 We put a partial order on the elements of $\mathbb{N}^{k}$ in the following way: we say that $n \leq m$ if and only if $ n_{i} \leq m_{i}$ for all $i$ , where $m = (m_{1}, m_{2}, . . . ,m_{k})$, and $n = (n_{1}, n_{2}, . . . , n_{k})$.
 We refer to elements of $Mor(\Lambda)$ as paths and write $r$ and $s$ for the domain and codomain maps respectively. 
 
 \par Our factorization property gives us that $d(\lambda)=0$ if and only if $\lambda = id_{v}$ for some $v \in Obj(\Lambda)$. We often refer to the elements of $Obj(\Lambda)$ as vertices. Thus we identify $Obj(\Lambda)$ with $\{ \lambda \in Mor(\Lambda) : d(\lambda) = 0 \}$, and write $\lambda \in \Lambda$ in place of $\lambda \in Mor(\Lambda)$.
 For any $\lambda \in \Lambda$ and $E \subset \Lambda$ we define $$ \lambda E := \{ \lambda \mu : \mu \in E, \ r(\mu) = s(\lambda) \}$$ and $$E \lambda := \{\mu \lambda : \mu \in E, \  s(\mu) = r(\lambda)\}.$$
By the factorization property we know that for any $l \leq m \leq n \in \mathbb{N}^{k}$ if $d(\lambda)=n$ then there are unique elements denoted $\lambda(0,l) ,\  \lambda(l,m),\  \lambda(m,n)$ of $\Lambda$ such that $d(\lambda (0,l))=l ,\  d(\lambda (l,m)) = m-l, \  d(\lambda(m,n))=n-m$ and \[\lambda = \lambda (0,l) \lambda(l,m) \lambda(m,n).\] 
For $m \in \mathbb{N}^{k}$ we define $\Lambda ^{m}= \{ \lambda \in \Lambda : d(\lambda) = m \}$. Thus $Obj(\Lambda) = \Lambda^{0}$.
We say that a $k$-graph is row-finite if for each $v \in \Lambda ^{0}$ and each $m \in \mathbb{N}^{k}$, the set $\Lambda^{m}(v)$ is finite. We say that $v \in \Lambda^{0}$ is a source if there exists $m \in \mathbb{N}^{k}$ such that $\Lambda^{m}(v) = \emptyset$.

As an example of a $k$-graph we offer up a common one that is used in many descriptions in research and later in this paper. We will use the notation from \cite{ClarkFlynnHuef}.
\begin{example}
For a fixed $m \in (\mathbb{N}\cup \{ \infty \})^{k}$: we define \[\Omega_{k,m}:= \{ (p,q) \in \mathbb{N}^{k} \times \mathbb{N}^{k} :p \leq q \leq m \}. \] The objects are $\Omega^{0}_{k,m} = \{ p \in \mathbb{N}^{k} : p \leq m \} $, and range and source maps $r(p,q)=p$ and $s(p,q)=q$. The morphisms $(p,q)$ and $(r,s)$ are composable if and only if $q=r$. When they are composable we have $(p,q)(q,s)=(p,s)$. The factorization property is fulfilled by $d:\Omega^{k,m} \mapsto \mathbb{N}^{k}$ defined by $d((p,q))=q-p$. So we have that the pair $(\Omega_{k,m},d)$ is a $k$-graph.
\end{example}
We introduce the set $\Lambda ^{\leq n}$ consisting of paths $\lambda$ with $d(\lambda) \leq n$ which cannot be extended to paths $\lambda \mu$ with $d(\lambda) < d(\lambda \mu) \leq n$. Thus 
\[\Lambda ^{\leq n} := \{ \lambda \in \Lambda :d(\lambda) \leq n, \text{ and } d(\lambda)_{i} < n_{i} \text{ implies } s(\lambda)\Lambda^{e_{i}}= \emptyset \}. \] We have that $v\Lambda^{\leq n} := v\Lambda \cap \Lambda^{\leq n} $ for $v \in \Lambda ^{0}$ is always nonempty.

We will be interested in row-finite locally-convex $k$-graphs throughout the paper, which were introduced in \cite{ClarkFlynnHuef} and \cite{RaeburnSimsYeend} 
\begin{definition}
We say $\Lambda$ is \emph{locally-convex} if for every $v\in \Lambda^{0}, 1\leq i$, $j\leq k$ with $i\neq j$, $ \lambda \in v \Lambda^{e_{i}}$ and $  \mu \in v \Lambda ^{e_{j}}$, the sets $s(\lambda)\Lambda^{e_{j}}$ and $s(\mu ) \Lambda ^{e_{i}}$ are nonempty.
\end{definition}
It is worth noting that by \cite[Remark 3.10]{RaeburnSimsYeend} every row-finite $1$-graph is locally-convex, as are all row-finite, higher-rank graphs with no sources.

 As we hope to give a vertex description of the regular ideals, it is natural to introduce some vertex sets which will help us classify ideals in our respective $k$-graph algebras.

\begin{definition} 
Let $(\Lambda, d)$ be a locally-convex row-finite $k$-graph. We say a subset $H$ of $\Lambda^{0}$ is \emph{hereditary} if $\lambda \in \Lambda$ and $r(\lambda) \in H$ imply $s(\lambda) \in H$. We say that $H$ is \emph{saturated} if for $v \in \Lambda^{0}, \  s(v\Lambda^{\leq e_{i}}) \subset H$ for some $i \in \{ 1, . . ..., k \}$ implies $v \in H$.
\end{definition}

As final definitions we introduce boundary paths of locally-convex, row-finite $k$-graphs and Condition ($B$). 
Boundary paths are used to put certain conditions on our $k$-graphs including aperiodicity in \cite{Kumjian Pask} and Condition $(C)$ of \cite{Sims}. We use it here to define Condition ($B$). Aperiodicity was a condition that Kumjian and Pask used along with cofinality to find a simplicity condition of row-finte, no source $k$-graph $C^{*}$-algebras. It served as an analogue to Condition ($L$) of directed graphs. When switching to row-finite locally-convex $k$-graph $C^{*}$-algebras, Raeburn, Sims, and Yeend made Condition ($B$) as an analogue of aperiodicity to allow for sources. With this they were able to prove the Cuntz-Krieger Uniqueness Theorem. By Theorem 8.4 of \cite{ClarkFlynnHuef} when a row-finite $k$-graph has no sources, Condition $(B)$ is equivalent to the aperiodicity condition. 

\begin{definition}
Let $\Lambda$ be a locally convex $k$-graph. A \emph{boundary path in $ \Lambda$} is a graph morphism $x: \Omega _{k,m} \mapsto \Lambda$ for some $m \in (\mathbb{N} \cup \infty)^{k}$ such that, whenever $v \in Obj(\Omega_{k.m})$ satisfies $v(\Omega_{k.m})^{\leq e_{i}} = \{ v \}$, we also have that $x(v)\Lambda^{\leq e_{i}} = \{ x(v) \}$. We denote the collection of all boundary paths in $\Lambda$ by $\Lambda ^{\leq \infty}.$ The range map of $\Lambda$ extends naturally to $\Lambda ^{\leq \infty}$ via $r(x) := x(0).$ For $v \in \Lambda^{0}$, we write $v\Lambda^{\leq \infty}$ for $\{x \in \Lambda^{\leq \infty} : r(x) = v \} $.

\end{definition}
If $\Lambda$ has no sources, then $\Lambda ^{\leq \infty} = \Lambda ^{\infty}$. It should also be noted that a boundary path can be composed with finite paths. That is if $x$ is a boundary path and $ \lambda \in \Lambda$ is a finite path with $s(\lambda) = r(x) $ then we define $ \lambda x: \Omega_{k, m+d(\lambda)} \mapsto \Lambda $ such that $$ \lambda x (d(\lambda)) = r(x), \ \lambda x (0) = r(\lambda),$$  $$\lambda x (l, l + e_{i}) = \lambda(l, l+e_i) \text{ if } l+e_{i} \leq d(\lambda),$$ $$\lambda x (d(\lambda) + l, d(\lambda) + l + e_{i}) = x(l + l+ e_{i}) \text{ if } l+e_{i} \leq m.$$ The rest of the graph morphism can be obtained by concatenating paths of length $e_{i}$. We have that $\lambda x$ is also a boundary path. For more information the reader can refer to \cite{ClarkFlynnHuef} and \cite{RaeburnSimsYeend}. 
\begin{definition}
We say that a vertex $v$ in a $k$-graph, $\Lambda$, satisfies Condition ($B$) if; \[ \text{ there exists } x\in v\Lambda ^{\leq \infty} \text{ such that } \alpha \neq \beta \in \Lambda \text{ implies } \alpha x \neq \beta x \]
We say that a $k$-graph, $\Lambda$, satisfies condition $(B)$ if every vertex in $\Lambda$ satisfies Condition $(B)$.
\end{definition}

\section{Background: Kumjian-Pask Algebras}
As Kumjian-Pask algebras are an algebraic generalization of $k$-graph $C^{*}$-algebras, we introduce ghost paths to take the place of adjoints. 
We will again be using the definitions and ideas from \cite{ClarkFlynnHuef}.
\begin{definition}

Define $G(\Lambda):= \{ \lambda^{*}:\lambda \in \Lambda \}$, and call each $\lambda^{*}$ a \emph{ghost path}. If $v\in \Lambda^{0}$, then we identify $v$ and $v^{*}$. We extend the degree functor $d$ and the range and source maps $r$ and $s$ to $G(\Lambda)$ by $d(\lambda^{*})=-d(\lambda)$, $r(\lambda^{*})=s(\lambda)$ and $s(\lambda^{*})=r(\lambda)$. We extend the factorization property to the ghost paths by setting $(\mu \lambda)^{*}= \lambda^{*} \mu^{*}$. We denote by $\Lambda^{\neq 0}$ the set of paths which are not vertices and by $G(\Lambda^{\neq 0})$ the set of ghost paths that are not vertices.
\end{definition}

\begin{definition}
Let $ \Lambda$ be a row-finite $k$-graph and let $R$ be a commutative ring with 1. A Kumjian-Pask $\Lambda$-family $(P, S)$ in an $R$-algebra $A$ consists of two functions $P: \Lambda ^{0} \mapsto A$ and $S: \Lambda ^{\neq 0} \cup G(\Lambda ^{\neq 0}) \mapsto A$ such that:
\begin{enumerate}
    \item[(KP1)] $\{P_{v}:v \in \Lambda^{0}\}$ is a family of mutually orthogonal idempotents;
    \item[(KP2)] for all $\lambda, \mu \in \Lambda ^{\neq 0}$ with $r(\mu) =s(\lambda)$, we have $S_{\lambda}S_{\mu} = S_{\lambda \mu} ,\  S_{\mu^{*}}S_{\lambda^{*}}= S_{(\lambda \mu)^{*}},\  P_{r(\lambda)}S_{\lambda}= S_{\lambda}=S_{\lambda}P_{s(\lambda)},  
    \text{ and } P_{s(\lambda)}S_{\lambda ^{*}}=S_{\lambda ^{*}}=S_{\lambda ^{*}}P_{r(\lambda)}$;
    \item[(KP3)]  for all $n\in \mathbb{N}^{k} \setminus \{0\}$ and $\lambda , \mu \in \Lambda^{\leq n}$, we have $S_{\lambda ^{*}}S_{\mu} = \delta_{\lambda , \mu } P_{s(\lambda)}$;
    \item[(KP4)] for all $v \in \Lambda^{0} $ and all $n \in \mathbb{N}^{k} \setminus \{ 0 \}$, we have $P_{v} = \sum _{ \lambda \in v \Lambda ^{\leq n}}S_{\lambda}S_{\lambda ^{*}}$.
\end{enumerate}
\end{definition}

\begin{theorem}
\cite[Proposition 3.3]{ClarkFlynnHuef}\label{7}. Let $\Lambda$ be a locally-convex, row-finite $k$-graph, $(P,S)$ a Kumjian–Pask $\Lambda$-family in an $R$-algebra $A$, and $\lambda, \mu \in \Lambda$. If $n \in \mathbb{N}^{k}$ such that $d(\lambda), \ d(\mu) \leq n$, then $S_{\lambda^{*}}S_{\mu}=\sum _{\lambda \alpha = \mu \beta , \lambda \alpha \in \Lambda ^{\leq n}} S_{\alpha}S_{\beta^{*}}$.
\end{theorem}
\begin{definition}
We define \emph{$KP_{R}(\Lambda)$} to be the universal $R$-algebra generated by a Kumjian–Pask $\Lambda$-family $(p,s)$, in the sense that if $(Q,T)$ is a Kumjian–Pask $\Lambda$-family in an $R$-algebra $A$, then there exists a $R$-algebra homomorphism $\pi_{Q,T} : KP_{R}(\Lambda) \mapsto A $ such that $\pi_{Q,T} \circ p = Q$ and $\pi_{Q,T} \circ s = T$. For every $r\in R \setminus \{ 0 \} $ and $v \in \Lambda^{0}$, we have $rp_{v} \neq 0$. 
\end{definition}
Grading plays an important role in the lattice of ideals of Kumjian-Pask algebras. We take time here to define gradings, as it will pertain to later theorems.
\begin{definition}
Let $G$ be an additive abelian group. A ring $A$ is \emph{$G$-graded} if there are additive subgroups $\{ A_{g}:g \in G \} $ of $A$ such that $A_{g}A_{h} \subset A_{g+h} $ and every nonzero $a\in A$ can be written in exactly one way as a finite sum $\sum _{g\in F} a_{g} $ of nonzero elements $a_{g} \in A_{g}$. The elements of $A_{g}$ are \emph{homogeneous of degree $g$}, and $a=\sum _{g \in F} a_{g} $ is the \emph{homogeneous decomposition of $a$}.
\end{definition}
Suppose that $A$ is $G$-graded by $\{A_{g}:g \in G\}$.  An ideal $I$ in $A$ is a \emph{graded ideal} if $\{ I \cap A_{g}:g \in G \}$ is a grading of $I$. Every ideal $I$ which is generated by a set $S$ of homogeneous elements is graded.
The following theorem allows us to put a $\mathbb{Z}^{k}$ grading on $KP_{R}(\Lambda)$ as well as giving us a full description of the algebra.

\begin{theorem}
\cite[Theorem 3.7]{ClarkFlynnHuef}\label{8}. Let $ \Lambda$ be a locally-convex, row-finite $k$-graph.
\begin{enumerate}
\item[(i)] There is a unique $R$-algebra $KP_{R}(\Lambda)$, generated by a Kumjian–Pask $\Lambda$-family $(p,s)$, such that if $(Q,T)$ is a Kumjian–Pask $\Lambda$-family in an $R$-algebra $A$, then there exists a unique $R$-algebra homomorphism $\pi_{Q,T} : KP_{R}(\Lambda) \mapsto A $ such that $\pi_{Q,T} \circ p = Q$ and $\pi_{Q,T} \circ s = T$. For every $r\in R \setminus \{ 0 \} $ and $v \in \Lambda^{0}$, we have $rp_{v} \neq 0$. 
\item[(ii)] The subsets $KP_{R}(\Lambda)_{n} := span \{ s_{\alpha}s_{\beta^{*}} : d(\alpha)-d(\beta)= n \}$ form a $\mathbb{Z}^{k}$-grading of $KP_{R}(\Lambda)$.
\end{enumerate}
\end{theorem}

By putting together Theorems \ref{7} and \ref{8} it can be seen that $KP_{R}(\Lambda) = span_{R} \{s_{\lambda}s_{\mu^{*}}: s(\lambda) = s(\mu)  \}$. 

\begin{definition}
Let $R$ be a ring. Let $\Lambda$ be a $k$-graph. Let $I$ be an ideal of the Kumjian-Pask Algebra $KP_{R}(\Lambda)$. We say that $I$ is \emph{basic} if it has the property such that if $rp_{v} \in I$ and $r \in R \setminus \{0\}$ then $p_{v} \in I$.
\end{definition}
\begin{remark}
We note that if $R$ is a field, all ideals  of $KP_{R}(\Lambda)$ are basic.
Indeed, let $I$ be an ideal. Suppose that $rp_{v} \in I$ and $r \neq 0$. Then \[rp_{v}r^{-1}p_{v} = rr^{-1}p_{v}p_{v}= p_{v} \in I.\]

\end{remark}
\begin{definition} 
For a subset $H$ of $\Lambda^{0}$ define \[ I(H) := span_{R}\{s_{\alpha}s_{\beta ^{*}} : s(\alpha) = s(\beta) \in H \} .
\]
 
\end{definition}
$I(H)$ will be an ideal if $H$ is saturated and hereditary.

In order to have a clear picture of the lattice isomorphism introduced in Definition \ref{1} we give notations that will be concise to use throughout.
\begin{definition}
For an ideal $I$ in $KP_{R}(\Lambda)$ we define 
$H(I) : = \{ v\in \Lambda ^{0} :p_{v} \in I \}$
\end{definition}

\begin{lemma}
\cite[Lemma 9.2]{ClarkFlynnHuef}. 
Let $H$ be a hereditary, saturated subset of $\Lambda^{0}$, and $I(H)$ be the ideal of $KP_{R}(\Lambda)$ generated by $ \{ p_{v}:v \in H \}$. Then \[ I(H)=span \{ s_{\alpha}s_{\beta ^{*}}: \alpha, \beta \in \Lambda,s(\alpha)=s(\beta)\in H\}. \]
\end{lemma}

\begin{theorem}\label{1}
\cite[Theorem 9.4]{ClarkFlynnHuef}
Let $\Lambda$ be a row-finite, locally-convex $k$-graph. Let $R$ be a commutative ring with 1. Then the map $H \rightarrow I(H)$ is a lattice isomorphism from the set of saturated hereditary subsets of $ \Lambda^{0}$ onto the lattice of basic graded ideals of $KP_{R}(\Lambda)$.
\end{theorem}
\begin{remark}
The inverse of the lattice isomorphism described in the above theorem is the map $I \mapsto H(I)$.
\end{remark}

We now show that there is an isomorphism between the quotient algebra created by quotienting by a basic, graded ideal and the Kumjian-Pask algebra of the quotient graph. For $I$ an ideal we define $\Lambda \setminus H(I) $ to be the small category with objects $\Lambda^{0} \setminus H(I)$, and morphisms $ \{ \lambda \in \Lambda : r(\lambda) \text{ and } s(\lambda) \in \Lambda^{0} \setminus H(I) \}$, with the factorization property $d$ inherited from $(\Lambda, d)$.

\begin{proposition}\label{100}
Let $\Lambda$ be a locally-convex row-finite $k$-graph and $R$ a commutative ring with 1. Let $I$ be a basic graded ideal of $KP_{R}(\Lambda)$, and let $(q,t)$ and $(p,m)$ be the universal Kumjian-Pask families in $KP_{R}(\Lambda \setminus H(I))$ and $KP_{R}(\Lambda)$, respectively. Then there exists an isomorphism $\pi : KP_{R}(\Lambda \setminus H(I)) \mapsto  KP_{R}(\Lambda)/I$ such that \[ \pi(q_{v}) = p_{v} + I, \pi(t_{\lambda}) = m_{\lambda} + I, \text{ and } \pi(t_{\mu^{*}}) = m_{\mu^{*}} + I    \] for $v \in \Lambda^{0}/H(I) $ and $\lambda , \mu \in s^{-1}(\Lambda^{0}/H(I))$.
\end{proposition}

\begin{proof}
First note that $\Lambda \setminus H(I)$ is indeed a locally-convex $k$-graph. This is shown in the proof of \cite[Theorem 5.2]{RaeburnSimsYeend}.

Now we show that $ \{p_{v} + I, m_{\lambda} + I, m_{\mu^{*} + I} \} $ is a Kumjian-Pask $(\Lambda \setminus H(I))$ family. 
(KP 1) and (KP 2) hold as $(p, m)$ is a Kumjian-Pask $\Lambda$ family. To see (KP 3) and (KP 4) we show that $(\Lambda \setminus H(I)) ^{ \leq n} = \Lambda^{\leq n} \setminus \{ \lambda \in \Lambda : m_{\lambda} \in  I \}$. Take $n \in \mathbb{N} ^{k}$. Here we note that since $I$ is basic and graded it is generated by the set of idempotents of a hereditary and saturated set of vertices \cite[Theorem 9.4]{ClarkFlynnHuef}. Thus  $\{ \lambda \in \Lambda : m_{\lambda} \in  I \} = \{ \lambda \in \Lambda : s(\lambda) \in H(I) \}$.
To see that  $$(\Lambda \setminus H(I))^{\leq n} \subseteq \Lambda^{\leq n} \setminus \{ \lambda \in \Lambda : s(\lambda) \in H(I) \},$$ take $ \lambda \in (\Lambda \setminus H(I))^{\leq n} $.
Note that $s(\lambda) \notin H(I)$, so we need only show that $\lambda \in \Lambda ^{\leq n}$. 
Suppose that $d(\lambda) = n$. Then $ \lambda \in \Lambda ^{\leq n}. $
\par
Now suppose $d(\lambda) < n$. For notation, assume that $s(\lambda) = v$. Then for every $ i $ such that $d(\lambda)_{i} < n_{i}$, $ v (\Lambda \setminus H(I))^{e_{i}} = \emptyset.$ Thus we must show that for such an $i$, $v \Lambda^{e_{i}} = \emptyset$. To obtain a contradiction, suppose there exists an $i$ so that $ v \Lambda^{e_{i}} \neq \emptyset$. Then we have that $s( v \Lambda ^{e_{i}}) \subseteq H(I)$ is nonempty. Thus by definition of $ v (\Lambda)^{\leq n}$, we get $s( v \Lambda ^{e_{i}}) = s( v \Lambda ^{ \leq e_{i}})$. Which, since $H(I)$ is saturated, gives $v \in H(I) $ a contradiction. So it must be that $$(\Lambda \setminus H(I))^{\leq n} \subseteq \Lambda^{\leq n} \setminus \{ \lambda \in \Lambda : s(\Lambda) \in H(I) \}.$$

To see the other direction we simply note that by definition $(\Lambda \setminus H(I))^{ n} \subseteq \Lambda ^{n} $. So the inclusion is clear when $d(\lambda) = n$. When $d(\lambda) < n$, we note that since $(\Lambda \setminus H(I))^{e_{i}} \subseteq \Lambda^{e_{i}}$, it must be that $\lambda \in (\Lambda \setminus H(I))^{\leq n}.$
Now it is simple to show that $ \{p_{v} + I, m_{\lambda} + I, m_{\mu^{*}} + I \} $ satisfy (KP 3) and (KP 4). 
\par 
Thus by \cite[Theorem 3.7]{ClarkFlynnHuef} we know there is a homomorphism $\pi_{p + I, m + I}$ satisfying the equations stated. Since other generators of $KP_{R}(\Lambda)$ belong to $I$, the family $(p+I, m+ I)$ generates $KP_{R}(\Lambda) / I$ and $\pi$ is surjective. 
Suppose that $\pi (rq_{v}) = 0$ for some $r \in R \setminus \{ 0 \} $ and $ v \notin H(I)$. Then $ rp_{v} + I = \pi (rq_{v}) = 0$, so that $rp_{v} \in I$ and since $I$ is basic, $p_{v} \in I$ as well. And this implies that $ v \in H(I)$, a contradiction. Thus $ \pi (rq_{v}) \neq 0$ for all $r \in R \setminus \{ 0 \}. $ Since $I$ is graded, then $KP_{R}(\Lambda) / I$ is graded by $(KP_{R}(\Lambda) / I ) _{n} = q(KP_{R}(\Lambda)_{n}),$ where $q : KP_{R}(\Lambda) \mapsto KP_{R}(\Lambda) / I$ is the quotient map. 
 If $ \alpha, \beta \in (\Lambda \setminus H(I))$ with $d(\alpha) - d(\beta) = n \in \mathbb{Z}^{k}, $ then $\pi ( t_{\alpha}t_{\beta ^{*}}) = s_{\alpha}s_{\beta ^{*}} + I = q(s_{\alpha}s_{\beta^{*}}) \in q(KP_{R}(\Lambda ) _{n}) = (KP_{R}(\Lambda)/I)_{n}.$ Thus $\pi $ is graded and thus by the graded uniqueness theorem, \cite[Theorem 4.1]{ClarkFlynnHuef}, $\pi$ is injective.
\end{proof}

Our last two theorems will be helpful in classifying when a basic regular ideal must be graded.

\begin{theorem}\label{31}
Let $\Lambda$ be a row-finite, locally-convex $k$-graph. Let $J$ be a basic ideal in $KP_{R}(\Lambda)$, then $I(H(J))$ is the largest basic, graded ideal contained in $J$.
\end{theorem}
\begin{proof}
As $J$ is an ideal and $I(H(J))$ is the ideal generated by $ \{ p_{v} : p_{v} \in J \}$ it is clear that $I(H(J)) \subseteq J$. Further as $J$ is an ideal, $H(J)$ is hereditary and saturated by $(KP 2-4)$, thus $I(H(J))$ is basic and graded by Theorem \ref{1}. It remains to show that $I(H(J))$ is the largest. For this we note that by Theorem \ref{1} all basic, graded ideals are generated by the vertex idempotents of a saturated and hereditary set of vertices. As $I(H(J))$ contains all such idempotents in $J$, there can be no larger basic, graded ideal in $J$.  
\end{proof}

\begin{theorem}\label{33}
Let $\Lambda$ be a row-finite, locally-convex $k$-graph. Let $J$ be a basic ideal in $KP_{R}(\Lambda)$, if $  \Lambda \setminus H(J) $ satisfies Condition $(B)$ then $J$ is basic and graded.
\end{theorem}
\begin{proof}
By Proposition \ref{100} we can identify $KP_{R}(\Lambda)/ I(H(J))$ with $KP_{R}(\Lambda \setminus H(J))$. We have that $I(H(J)) \subseteq J$. Let $N$ be the image of $J$ under the quotient map $ KP_{R}(\Lambda) \mapsto KP_{R}(\Lambda) / I(H(J)) \cong KP_{R}(\Lambda \setminus H(J))$. Then \[ KP_{R}(\Lambda) / J \cong KP_{R}(\Lambda \setminus H(J)) / N .\] 
Consider our quotient mapping $q:KP_{R}(\Lambda \setminus H(J)) \mapsto KP_{R}(\Lambda \setminus H(J)) / N.$ Note that $H(N) \subseteq (\Lambda \setminus H(J))^{0}$ is empty. Thus $q(p_{v}) \neq 0$ for all $ v \in \Lambda \setminus H(J)^{0}$. Since $\Lambda \setminus H(J)$ satisfies Condition $(B)$, we have by \cite[Theorem 4.2]{ClarkFlynnHuef} that the quotient map is injective. Hence $N$ is trivial and $J = I(H(J))$.\end{proof}

\section {Regular Ideals of Kumjian-Pask Algebras}
We start off this section by giving a series of observations which will help us to our goal of giving a vertex description of the basic, regular, graded ideals of $KP_{R}(\Lambda)$. 

\begin{definition}
An ideal $J$ in an algebra $A$ is called regular if $J^{\perp \perp} = J$ where $J^{\perp} = \{a \in A : ax = xa = 0$ $\forall$ $x \in J\}$.
\end{definition}
We note that if $J$ is an ideal then $J^{\perp}$ is a regular ideal.
The proof of the following lemma is largely the same as Lemma 3.2 of \cite{GoncalvesRoyer} but we include it here for completeness.
\begin{lemma}\label{101}
Let $\Lambda$ be a row-finite locally-convex $k$-graph. Let $R$ be a commutative ring with 1. If $J$ is a graded ideal of $KP_{R}(\Lambda)$ then $J^{\perp}$ is a graded ideal.
\end{lemma}
\begin{proof}
Let $z=\sum_{n\in \mathbb{Z}^{k}} r_{n}z_{n} \in J^{\perp}$. We need to prove that each $r_{n}z_{n} \in J^{\perp}$. Let $x \in J$. Since $J$ is graded then $x=\sum_{n\in \mathbb{Z}^{k}} t_{n} x_{n}$, where each $t_{n}x_{n}$ is homogeneous and $t_{n} x_{n} \in J$. So it is enough to check that $r_{n}z_{n}t_{i}x_{i}=t_{i}x_{i}r_{n}z_{n}= 0$ for each $i, n$.
Since $z\in J^{\perp}$, and $t_{i}x_{i} \in J$ for each fixed $i\in \mathbb{Z}^{k}$, we have that $\sum_{n\in \mathbb{Z}^{k}} t_{i}x_{i}r_{n}z_{n} = t_{i}x_{i}z= 0 =zt_{i}x_{i}= \sum_{n\in \mathbb{Z}^{k}}r_{n}z_{n}t_{i}x_{i}$, and hence the grading (and the fact that $t_{i}x_{i}$ is homogeneous) implies that $r_{n}z_{n}t_{i}x_{i} = t_{i}x_{i}r_{n}z_{n} = 0$ as desired.
\end{proof}

\begin{lemma}\label{40}
Let $\Lambda$ be a row-finite, locally-convex $k$-graph. Let $R$ be a commutative ring with 1. If $J$ is a basic graded ideal of $KP_{R}(\Lambda)$ then $J^{\perp}$ is a basic graded ideal.
\end{lemma}
\begin{proof}
We have that $J^{\perp}$ is graded by Lemma \ref{101}. It remains to show that it is basic. As $J$ is a graded basic ideal it must be that $J = I(H) $ for some saturated hereditary set $H$ by Theorem \ref{1}. Suppose that $rp_{v} \in H(J^{\perp})$ then we have that $rp_{v}s_{\alpha}s_{\beta^{*}} = 0$ and $s_{\alpha}s_{\beta^{*}}rp_{v} = rs_{\alpha}s_{\beta^{*}}p_{v} = 0$ for all $s_{\alpha}s_{\beta^{*}} \in I(H)$. But this is true if and only if $v \neq r(\alpha) $ and $v \neq r(\beta)$ for all $s_{\alpha}s_{\beta^{*}} \in I(H)$. Thus we also get that $p_{v}s_{\alpha}s_{\beta^{*}} = s_{\alpha}s_{\beta^{*}}p_{v} = 0$ for all $s_{\alpha}s_{\beta^{*}} \in I(H)$. Thus $p_{v} \in J^{\perp}$.
\end{proof}
The following notation will be useful for the remainder of the paper as they reoccur. 
\begin{definition}
\begin{enumerate}
    \item[(i)] For $w \in \Lambda^{0}$, put $  T(w) = \{ s(\lambda) : \lambda \in \Lambda, r(\lambda) = w \} $ .
    \item[(ii)] If $ I \subseteq KP_{R}(\Lambda)$ an ideal, let  $\overline{H}(I) \subseteq \Lambda^{0}$ be the set \[ \overline{H}(I) = \{ r(\lambda) : \lambda \in \Lambda \text{ and } s(\lambda) \in H(I) \} .\]
\end{enumerate} 
\end{definition}

We are now ready to describe the vertex set of $J^{\perp}$ and in turn give a vertex description of $J^{\perp}$.
\begin{lemma}\label{2}

Let $\Lambda$ be a row finite, locally-convex $k$-graph. Let $H$ be a hereditary and saturated subset of $\Lambda^{0}$. Let $J$ be the ideal generated by $H$. Then $$ H(J^{\perp}) = \{v \in \Lambda^{0} : v \Lambda H = \emptyset\} = \Lambda^{0} \setminus \bar{H}(J).$$
\end{lemma}
\begin{proof}
Since we define $H(J^{\perp})$ to be $\{ v \in \Lambda^{0} : p_{v} \in J^{\perp} \}$,
  $$ H(J^{\perp}) = \{v \in \Lambda : p_{v}s_{\alpha}s^{*}_{\beta}= s_{\alpha}s^{*}_{\beta}p_v = 0 \text{ for all } s_{\alpha}s^{*}_{\beta} \in J \}$$
So $v \in H(J^{\perp})$ if and only if for all $s_{\alpha}s^{*}_{\beta} \in J$:
\begin{enumerate}
    \item[(i)] $p_{v}s_{\alpha}s_{\beta}^{*}=0$; and 
    \item[(ii)] $s_{\alpha}s^{*}_{\beta} p_{v}=0$.  
\end{enumerate}

 For (i) to be true we require that $v \neq r(\alpha)$. So we require that $ v \neq r(\alpha)$ for all $\alpha $ with $s_{\alpha}s_{\beta}^{*} \in J$. So for $v \in H(J^{\perp})$ there can be no path from a vertex in $H$ to $v$ (as $s_{\alpha} = s_{\alpha}p_{s(\alpha)} \in J)$.
\\For (ii) to be true we require that $v \neq r(\beta)$. So we require that $v \neq r(\beta)$ for all $\beta$ with $s_{\alpha}s_{\beta}^{*} \in J$. So for $v \in H(J^{\perp})$ there can be no path from a vertex in $H$ to $v$.

So we have the following description:
$$ H(J^{\perp}) = \{v \in \Lambda^{0} : \text{ there is no path from a vertex in } H \text{ to } v \}$$ as desired.
\end{proof}

In \cite{GoncalvesRoyer} a vertex set description was given for the regular ideals of row-finite, no source Leavitt Path Algebras. As Leavitt Path Algebras are isomorphic to Kumjian-Pask Algebras generated by a $1$-graph, the following result generalizes \cite{GoncalvesRoyer} both in moving to higher-rank and by allowing for sources.
\begin{theorem}\label{3}
Let $\Lambda$ be a locally-convex, row-finite $k$-graph. Let $R$ be a commutative ring with 1. Let $J \subseteq KP_{R}(\Lambda)$ be a basic graded ideal. Then: 
\begin{enumerate}
    \item[(i)] $J^{\perp} = I( {\Lambda^0 \setminus \bar{H} (J)});$
    \item[(ii)] $J^{\perp \perp} = I({ \{ w \in \Lambda ^{0} : T(w) \subseteq \bar{H}(J) \}  });$ and
    \item[(iii)] $J$ is regular if and only if $H(J) = \{ w \in \Lambda ^{0} : T(w) \subseteq \bar{H}(J) \}$
\end{enumerate}
\end{theorem}
\begin{proof}
We know that $J = I(H)$ for some saturated and hereditary set $H$ by Theorem \ref{1}. Thus from Lemma \ref{2} we know $H(J^{\perp}) = \Lambda^{0} \setminus \bar{H}(J)$. We also know that $J^{\perp}$ is basic and graded since $J$ is basic, and graded by Lemma \ref{40}. So it must be generated by $\{ p_{v} : v \in H(J^{\perp}) \}$. This proves (i). The rest follows.
\end{proof}
\begin{corollary}
Let $\Lambda$ be a locally-convex, row-finite $k$-graph. Let $R$ be a field. Let $J \subseteq KP_{R}(\Lambda)$ be a graded ideal. Then: 
\begin{enumerate}
    \item[(i)] $J^{\perp} = I({\Lambda^0 \setminus \bar{H} (J)});$
    \item[(ii)] $J^{\perp \perp} = I({ \{ w \in \Lambda ^{0} : T(w) \subseteq \bar{H}(J) \}  } );$ and
    \item[(iii)] $J$ is regular if and only if $H(J) = \{ w \in \Lambda ^{0} : T(w) \subseteq \bar{H}(J) \}$
\end{enumerate}
\end{corollary}
\begin{proof}
All ideals are basic as $R$ is a field. The rest follows from above.
\end{proof}

We now show that quotienting by a basic, graded, regular ideal of a Kumjian-Pask Algebra preserves Condition $(B)$. In was shown in \cite{GoncalvesRoyer} that Condition $(L)$ is preserved when quotienting by a basic, graded, regular ideal of a Leavitt Path Algebra.

\begin{theorem}\label{5}
Suppose that $\Lambda$ is a row-finite, locally-convex $k$-graph which satisfies Condition ($B$). 
If $J$ is a regular, basic, graded ideal, then $\Lambda \setminus J$ satisfies Condition $(B)$.

\end{theorem}
\begin{proof}
First note that using Lemma \ref{2} and replacing $J$ with $J^{\perp}$, and since $J$ is regular, basic and graded, we have that $(\Lambda/J)^{0} = \Lambda ^{0}\setminus H(J) = \overline{H}(J^{\perp})$. For a vertex $v \in H(J^{\perp})$ we know there exists an $x \in v \Lambda ^{\leq \infty}$ such that if $\alpha \neq \beta$ then $\alpha x \neq \beta x$. As $H(J^{\perp})$ is saturated and hereditary (since $J^{\perp}$ is an ideal) and $r(x) = v$ we know that $x(i, i) \in H(J^{\perp})$ for all $i$. Thus since $J ^{\perp}$ is an ideal we conclude that $x \in (\Lambda/J) ^{\leq \infty}$.
Hence all vertices in $H(J^{\perp}) $ satisfy Condition $(B)$. For a vertex $w$ in $\overline{H} (J ^{\perp})$ we know there exists a finite path $\gamma$ with $r(\gamma ) = w $ and $s(\gamma) \in H(J^{\perp})$. Therefore $\gamma x_{s(\gamma)} \in (\Lambda/J) ^{\leq \infty}$ satisfies Condition $(B)$ at $w$. To see this note that if $\alpha \neq \beta$ then  $\alpha \gamma \neq \beta \gamma$ and $x_{s(\gamma)}$ satisfies Condition $(B)$ at $s(\gamma)$.
\end{proof}

We remind the reader that for a row-finite $k$-graph with no sources that Condition $(B)$ is equivalent to the aperiodicity condition \cite[Lemma 8.4]{ClarkFlynnHuef}. The corollary follows immediately.

\begin{corollary}

Let $\Lambda$ be a row finite $k$-graph with no sources which is aperiodic. Let $J$ be a basic, graded regular ideal of $KP_{R}(\Lambda)$ then $\Lambda / J$ is aperiodic.
\end{corollary}
We finish the section with some theorems that allow us to show sufficient conditions for when a basic regular ideal must be graded.

\begin{lemma}\label{32}
Let $\Lambda$ be a locally-convex, row-finite $k$-graph. Let $J$ be a basic, regular ideal of $KP_{R}(\Lambda)$. Then $I(H(J)) \subseteq J$ is a regular basic, graded ideal.
\end{lemma}
\begin{proof}
We know that $I(H(J)) \subseteq J$ and that $I(H(J)) \subseteq I(H(J))^{\perp \perp} \subseteq J^{\perp \perp} = J.$ As $I(H(J))$ is basic and graded, by Lemma \ref{40} we have $I(H(J))^{\perp} $ and $I(H(J))^{\perp \perp} $ are basic and graded. By Theorem \ref{31}, $I(H(J))$ is the largest gauge-invariant ideal in $J$. Thus $I(H(J)) = I(H(J))^{\perp \perp}.$
\end{proof}

\begin{proposition}\label{34}
If $\Lambda$ is a locally-convex, row-finite $k$-graph satisfying Condition $(B)$, and $J$ is a basic, regular ideal in $KP_{R}(\Lambda)$, then $J$ is graded.
\end{proposition}
\begin{proof}
As $J$ is regular, we have that $I(H(J))$ is regular by Lemma \ref{32}. Thus by Theorem \ref{5}, $ \Lambda \setminus H(J)$ satisfies Condition $(B)$. It follows that $J$ is graded by Theorem \ref{33}. 
\end{proof}

Putting together Theorem \ref{40}, Proposition \ref{34} and Theorem \ref{5} that we get the following corollary.
\begin{corollary}\label{60}
 Let $\Lambda$ be a locally-convex, row-finite $k$-graph satisfying Condition $(B)$. Let $J$ be a basic, regular ideal in $KP_{R}(\Lambda).$ Then $\Lambda \setminus J$ satisfies Condition $(B)$ and $KP_{R}(\Lambda / J) \cong KP_{R}(\Lambda \setminus H(J)).$
\end{corollary}

\section{Background: $k$-Graph $C^{*}$-Algebras}
In the following section we will be giving background information and theorems to help us establish similar classification to the regular ideals in $k$-graph $C^{*}$-algebras. We begin by defining the Cuntz-Kreiger $\Lambda$ family for a $C^{*}$-algebra. 

\begin{definition}
Let $\Lambda$ be a row-finite $k$-graph. A \emph{Cuntz–Krieger $\Lambda$-family} in a $C^{*}$-algebra $B$ consists of a family of partial isometries $\{ s_{\lambda}:\lambda \in \Lambda \}$ satisfying the Cuntz–Krieger relations:
\begin{enumerate}
    \item[(KP1)] $\{ s_{v} :v\in \Lambda^{0} \}$ is a family of mutually orthogonal projections;
    \item[(KP2)]$s_{\lambda \mu }=s_{\lambda}s_{\mu} $ for all $\lambda , \mu \in \Lambda$ with $s(\lambda)=r(\mu)$;
    \item[(KP3)] $s^{*}_{\lambda} s_{\lambda} = s_{s(\lambda)}$;
    \item[(KP4)] $s_{v}=\sum _{\lambda \in \Lambda ^{\leq m}(v)} s_\lambda s^{*}_{\lambda}$ for all $v\in \Lambda ^{0} $ and $m\in \mathbb{N}^{k}$
\end{enumerate} 
\end{definition}
\begin{theorem}
\cite[Theorem 3.15.]{RaeburnSimsYeend} Let $(\Lambda, d)$ be a row-finite $k$-graph. Then there is a Cuntz–Krieger  $\Lambda$-family $ \{ s_{\lambda}:\lambda \in \Lambda \} $ with each $s_{\lambda}$ non-zero if and only if  $\Lambda$ is locally-convex.
\end{theorem}

Given a row-finite $k$-graph $(\Lambda, d)$, there is a $C^{*}$-algebra $C^{*}(\Lambda)$ generated by a universal Cuntz–Krieger $\Lambda$-family $\{s_{\lambda} : \lambda \in \Lambda \}$ \cite{RaeburnSimsYeend}. We call this algebra the $k$-graph $C^{*}$-algebra for $\Lambda$ and denote it $C^{*}(\Lambda)$.
\begin{theorem}\label{11}

\cite[Proposition 3.5]{RaeburnSimsYeend}\label{200}Let $(\Lambda, d)$ be a row-finite $k$-graph and let $\{ s_{\lambda} :\lambda \in \Lambda \} $ be a Cuntz–Krieger $\Lambda$-family. Then for $\lambda , \mu \in \Lambda$ and $q\in \mathbb{N}^{k}$ with $d(\lambda),d(\mu) \leq q$ we have \[s^{*}_{\lambda} s_{\mu} = \sum _{\lambda \alpha = \mu \beta, \lambda \alpha \in \Lambda ^{\leq q}} s_{\alpha}s_{\beta^{*}}\]
\end{theorem}

Hence Theorem \ref{200} gives us that $C^{*}(\Lambda) = \overline{span} \{ s_{\alpha} s_{\beta ^{*}}: s(\beta) = s(\alpha) \}$.

Similar to the graph $C^{*}$-algebra, the universality of $C^{*}(\Lambda)$ gives us an action of $\mathbb{T}^{k}$ on $C^{*}(\Lambda)$ known as the gauge action.

\begin{definition}
Let $(\Lambda, d)$ be a row-finite $k$-graph. For $z\in \mathbb{T}^{k}$ and $n\in \mathbb{Z}^{k}$, let $z^{n}:=z^{n_{1}}_{1}...z^{n_{k}}_{k}$. Then $\{ z^{d(\lambda)}s_{\lambda}:\lambda \in \Lambda \}$ is a Cuntz–Krieger $\Lambda$-family which generates $C^{*}(\Lambda)$, and the universal property of $C^{*}(\Lambda)$ gives a homomorphism $\gamma_{z}:C^{*} (\Lambda) \mapsto C^{*}(\Lambda)$ such that $\gamma_{z}(s_{\lambda})=z^{d(\lambda)}s_{\lambda}$ for $\lambda \in \Lambda$; $\gamma_{\overline{z}}$ is an inverse for $\gamma_{z}$, so $\gamma_{z}$ is an automorphism. This action is strongly continuous and known as the \emph{gauge action}.
\end{definition}
For an ideal $I$ in $C^{*}(\Lambda)$ we denote $H(I):= \{ v \in \Lambda^{0} : p_{v} \in I \}$.
\begin{theorem}\label{10}
\cite[Theorem 5.2]{RaeburnSimsYeend} Let $(\Lambda, d)$ be a locally-convex row-finite $k$-graph. For each subset $H$
of $\Lambda^{0}$, let $I(H)$ be the closed ideal in $C^{*}(\Lambda)$ generated by $ \{ s_{v} : v \in H \}$.
\begin{enumerate}
    \item[(i)] The map $H \mapsto I(H)$ is an isomorphism of the lattice of saturated hereditary subsets of $\Lambda^{0}$ onto the lattice of closed gauge-invariant ideals of $C^{*}(\Lambda)$.
    \item[(ii)] Suppose $H$ is saturated and hereditary.
    Then $ \Lambda \setminus H$, the small category with objects $\Lambda^{0} \setminus H$, and morphisms $ \{ \lambda \in \Lambda : r(\lambda) \text{ and } s(\lambda) \in \Lambda^{0} \setminus H(I) \}$, with the factorization property $d$ inherited from $(\Lambda, d)$.
    is a locally-convex row-finite $k$-graph, 
    and $C^{*}(\Lambda)/I(H)$ is canonically isomorphic to $C^{*} (\Lambda \setminus H)$.
    \item[(iii)] If $H$ is any hereditary subset of $ \Lambda^{0}$, then $\Lambda (H)$, the small category with objects $H$ and morphisms $\{\lambda \in \Lambda: r(\lambda) \in H \}$ and the factorization property $d$ inherited from $ \Lambda$, is a locally-convex row-finite $k$-graph,
    $C^{*}(\Lambda(H))$ is canonically isomorphic to the subalgebra $C^{*}({s_{\lambda} : r(\lambda) \in H})$ of $C^{*}(\Lambda)$, and this subalgebra is a full corner in $I(H)$.
\end{enumerate}

\end{theorem}
\begin{remark}
The inverse of the lattice isomorphism in $(i)$ is $I \mapsto H(I)$.
\end{remark}

\begin{remark}
By putting together Theorems \ref{10} and \ref{11} we get the following for $J$ an ideal, generated by a saturated and hereditary $H$ of a locally-convex, row-finite $k$-graph.  $$C^{*}(\Lambda) = \overline{span} \{ s_{\alpha}s_{\beta}^{*} : s(\alpha) = s(\beta) \},$$
Hence,
$$J = \overline{span} \{p_{v}s_{\alpha}s_{\beta}^{*}, s_{\alpha}s^{*}_{\beta}p_{v} : v \in H,\}$$
$$=\overline{span}\{s_{\alpha}s_{\beta}^{*}: r(\beta) \in H \text{ or } r(\alpha) \in H \} .$$
  Since we need that $s(\alpha) = s(\beta)$ for $s_{\alpha}s^{*}_{\beta} \neq 0$ we can conclude that for $s_{\alpha}s^{*}_{\beta} \in J$ if $r(\beta) \in H$ then $s(\alpha) \in H$ as $H$ is hereditary. Similarly, we get if $r(\alpha) \in H$ then $s(\beta) \in H$.
So we have:
$$J = \overline{span} \{ s_{\alpha} s^{*}_{\beta} :r(\beta) \in H \text{ or } r(\alpha) \in H, \text{ and } s(\alpha) = s(\beta) \in H \}$$
\end{remark}

We finish this section with two additional applications of Theorem \ref{10}.
\begin{theorem}\label{21}
Let $\Lambda$ be a locally-convex $k$-graph. Let $J$ be an ideal in $C^{*}(\Lambda)$, then $I(H(J))$ is the largest gauge-invariant ideal contained in $J$.
\end{theorem}

\begin{theorem}\label{23}
Let $\Lambda$ be a locally-convex row-finite $k$-graph. Let $J$ be an ideal in $C^{*}(\Lambda)$ if $\Lambda \setminus H(J) $ satisfies Condition $(B)$ then $J$ is gauge invariant.
\end{theorem}

The proofs follow the same reasoning as in the Kumjian-Pask algebra case.

\section{Regular Ideals Of $k$-Graph $C^{*}$-Algebras}

In this section we give analagous proofs of those in Section 4 for the $k$-graph $C^{*}$-Algebras. As many of the proofs follow the same reasoning as the Kumjian-Pask Algebra case we omit them here when logical. We refer the reader to Section 4 for the full detais.

As we hope to give a vertex description of the gauge-invariant regular ideals in $C^{*}(\Lambda)$ we first remind the reader of the definition of a regular ideal. 
\begin{definition}
An ideal $J$ in an algebra $A$ is called regular if $J^{\perp \perp} = J$ where $J^{\perp} = \{a \in A : ax = xa = 0$ $\forall$ $x \in J\}$.
\end{definition}
We also remind that if $J$ is an ideal then $J^{\perp}$ is a regular ideal.
We will show that $J^{\perp} $ must be gauge-invariant if $J$ is. This will give us a nice starting point for the vertex description of $J^{\perp}$ and in turn $J$.

\begin{lemma}\label{20}
Let $J$ be a gauge-invariant ideal in a $k$-graph $C^{*}$-algebra $C^{*}(\Lambda)$. Then $J^{\perp}$ is a gauge-invariant regular ideal.
\end{lemma}
\begin{proof}
For an ideal $J$ we know that $J^{\perp}$ is always a regular ideal.
It remains to show that it is gauge-invariant. Suppose that $a \in J^{\perp}$ then for any $z \in \mathbb{T}^{k}$ we have that:
$$ \{ \gamma_{z}(a)b : b \in J \} = \{ \gamma_{z}(a)\gamma_{z}(\gamma_{\overline{z}}(b)) : b \in J \} = \{\gamma_{z}(a \gamma_{\overline{z}}(b)) : b \in J \} = \{0\}$$
as $\gamma_{\overline{z}}(b) \in J$ since $J$ gauge-invariant.
Similarly:
$$ \{ b \gamma_{z}(a): b \in J \} = \{ \gamma_{z}(\gamma_{\overline{z}}(b)) \gamma_{z}(a) : b \in J \} = \{ \gamma_{z}(\gamma_{\overline{z}}(b)a) : b \in J \} = \{ 0 \}.$$
So we have $\gamma_{z}(a) \in J$.

\end{proof}
We give now the $C^{*}$-algebra definitions that are analogues of the ones used in the regular ideal section for Kumjian-Pask algebras:

\begin{definition}
\begin{enumerate}
    \item[(i)] For $w \in \Lambda^{0}$, put $T(w) = \{ s(\lambda) : \lambda \in \Lambda, r(\lambda) = w \}$.
    \item[(ii)] If $ I \subseteq C^{*}(\Lambda)$ an ideal, let  $\overline{H}(I) \subseteq \Lambda^{0}$ be the set $\overline{H}(I) = \{ r(\lambda) : \lambda \in \Lambda$ and $s(\lambda) \in H(I) \}$.
\end{enumerate} 
\end{definition}

\begin{lemma}\label{50}
Let $\Lambda$ be a row finite, locally-convex $k$-graph. Let $H$ be a hereditary and saturated subset of $\Lambda^{0}$. Let $J$ be the ideal generated by $H$. Then $$ H(J^{\perp}) = \{v \in \Lambda^{0} : v \Lambda H = \emptyset\}=\Lambda^{0} \setminus \bar{H}(J).$$
\end{lemma}
\begin{proof}

Since we define $H(J^{\perp})$ to be $\{ v \in \Lambda^{0} : p_{v} \in J^{\perp} \}$,
  $$ H(J^{\perp}) = \{v \in C^{*}(\Lambda) : p_{v}s_{\alpha}s^{*}_{\beta}= s_{\alpha}s^{*}_{\beta}p_v = 0 \text{ for all } s_{\alpha}s^{*}_{\beta} \in J \}$$

The rest of the proof follows a similar reasoning to the Kumjian-Pask case.
\end{proof}

\begin{theorem}\label{4}
Let $\Lambda$ be a locally-convex, row-finite $k$-graph. Let $J \subseteq C^{*}(\Lambda)$ be a gauge-invariant ideal. Then: 
\begin{enumerate}
    \item[(i)] $J^{\perp} = I( \Lambda^0 \setminus \bar{H} (J) )$;
    \item[(ii)] $J^{\perp \perp} = I ( \{ w \in \Lambda ^{0} : T(w) \subseteq \bar{H}(J) \}  );$ and
    \item[(iii)] $J$ is regular if and only if $H(J) = \{ w \in \Lambda ^{0} : T(w) \subseteq \bar{H}(J) \}$.
\end{enumerate}
\end{theorem}
\begin{proof}
From \cite{RaeburnSimsYeend} [Theorem 5.2] we know that $J$ must be generated by $\{ p_{v} :v \in H \}$ for some saturated and hereditary set $H$. Thus by Lemma \ref{50}, We know that $H(J^{\perp}) = \Lambda^{0} \setminus \bar{H}(J)$. But from Lemma \ref{20} we know that $J^{\perp}$ is gauge-invariant. So using \cite{RaeburnSimsYeend} [Th. 5.2] it must be that $J^{\perp}  = I (\Lambda^0 \setminus \bar{H} (J))$. Proving i. The rest follow. 
\end{proof}

\begin{theorem}\label{6}
Suppose that $\Lambda$ is a row-finite, locally-convex $k$-graph which satisfies Condition ($B$). 
 If $J$ is a regular gauge-invariant ideal, then $ \Lambda \setminus J$ satisfies Condition $(B)$.

\end{theorem}

The proof follows a similar reasoning as the Kumjian-Pask algebra case since the regular ideals of both have the same vertex description and satisfying Condition $(B)$ is a property of the graph. We further remind the reader that for a row-finite $k$-graph without sources, satisfying Condition $(B)$ is equivalent to being aperiodic \cite[Lemma 8.4]{ClarkFlynnHuef}, So the corollary follows.
\begin{corollary}
 Let $\Lambda$ be a row finite $k$-graph with no sources which is aperiodic. Let $J$ be a gauge-invariant regular ideal of $C^{*}(\Lambda)$ then $ \Lambda / J $ is aperiodic.
\end{corollary}

\begin{lemma}\label{22}
Let $\Lambda$ be a locally-convex, row-finite $k$-graph. Let $J$ be a regular ideal of $C^{*}(\Lambda)$. Then $I(H(J)) \subseteq J$ is a regular gauge-invariant ideal.
\end{lemma}

The proof is again similar to the Kumjian-Pask algebra case.

\begin{proposition}\label{24}
If $\Lambda$ is a locally-convex, row-finite $k$-graph satisfying Condition $(B)$, and $J$ is a regular ideal in $C^{*}(\Lambda)$, then $J$ is gauge-invariant.
\end{proposition}
\begin{proof}
As $J$ is regular, we have that $I(H(J))$ is regular by lemma \ref{22}. Thus by Theorem \ref{6} $ \Lambda \setminus H(J)$ satisfies Condition $(B)$. It follows that $J$ is gauge invariant by Theorem \ref{23}. 
\end{proof}

Putting together Theorem \ref{10}, Proposition \ref{24} and Theorem \ref{6} that we get the following Corollary.
\begin{corollary}\label{61}
 Let $\Lambda$ be a row-finite, locally-convex $k$-graph satisfying Condition $(B)$. Let $J$ be a regular ideal in $C^{*}(\Lambda).$ Then $\Lambda \setminus J$ satisfies Condition $(B)$ and $C^{*}(\Lambda) / H(J) \cong C^{*}(\Lambda \setminus H(J)).$
\end{corollary}

\end{document}